\documentclass[10pt]{amsart}
\usepackage{amssymb}
\usepackage{epsfig}
\usepackage{url}
\usepackage{setspace}
\usepackage{color}
\theoremstyle{plain}

\newtheorem{thm}{Theorem}[section]
\newtheorem{cor}[thm]{Corollary}
\newtheorem{lem}[thm]{Lemma}

\newtheorem{rem}[thm]{Remark}
\newtheorem{ques}[thm]{Question}
\newtheorem{conj}[thm]{Conjecture}

\def\bbb{\mathbb}
\def\op{\operatorname}
\newcommand{\eps}{\varepsilon}
\renewcommand{\phi}{\varphi}

\newcommand{\N}{\bbb{N}}

\begin{document}
\title[On $p$-adic valuations of colored $p$-ary partitions]{On $p$-adic valuations of colored $p$-ary partitions }
\author{Maciej Ulas and B{\l}a\.{z}ej \.{Z}mija}

\thanks{The research of the first author is supported by the grant of the Polish National Science Centre no. UMO-2012/07/E/ST1/00185. During the preparation of the work, the second author was a scholarship holder of the Kartezjusz program funded by the Polish National Center for Research and Development.}

\begin{abstract}
Let $m\in\N_{\geq 2}$ and for given $k\in\N_{+}$ consider the sequence $(A_{m,k}(n))_{n\in\N}$ defined by the power series expansion
$$
\prod_{n=0}^{\infty}\frac{1}{\left(1-x^{m^{n}}\right)^{k}}=\sum_{n=0}^{\infty}A_{m,k}(n)x^{n}.
$$
The number $A_{m,k}(n)$ counts the number of representations of $n$ as sums of powers of $m$, where each summand has one among $k$ colors. In this note we prove that for each $p\in\mathbb{P}_{\geq 3}$ and $s\in\N_{+}$, the $p$-adic valuation of the number $A_{p,(p-1)(p^s-1)}(n)$ is equal to 1 for $n\geq p^s$. We also obtain some results concerning the behaviour of the sequence $(\nu_{p}(A_{p,(p-1)(up^s-1)}(n)))_{n\in\N}$ for fixed $u\in\{2,\ldots,p-1\}$ and $p\geq 3$. Our results generalize the earlier findings obtained for $p=2$ by Gawron, Miska and the first author.

\end{abstract}

\maketitle

\section{Introduction}\label{sec1}
Let $\N$ denotes the set of nonnegative integers and $\mathbb{P}$ be the set of prime numbers. Let $m\in\N_{\geq 2}$ and consider the sequence $(a_{m}(n))_{n\in\N}$ defined by the power series expansion of the infinite product
$$
F_{m}(x)=\prod_{n=0}^{\infty}\frac{1}{1-x^{m^{n}}}=\sum_{n=0}^{\infty}a_{m}(n)x^{n}.
$$
From the general theory of partitions (see the first chapter in the book of Andrews \cite{And}) we know that the sequence $(a_{m}(n))_{n\in\N}$ has a natural combinatorial interpretation. Indeed, we have $a_{m}(0)=1$ and for $n\in\N_{+}$, the number $a_{m}(n)$ counts the number of so called {\it $m$-ary representations} of a non-negative integer $n$, i.e., representations of the form
\begin{equation}\label{rep}
n=\sum_{i=0}^u \eps_i m^i,
\end{equation}
where $u\in\N$ and $\eps_i\in\N$ for $i\in\{0,\ldots,u\}$. Let us note that the $m$-ary representation is generalization of the (unique) representation of the integer $n$ in the base $m$, i.e., the representation (\ref{rep}) with $\eps_{i}\in\{0,\ldots,m-1\}$.

In the sequel we will need a notion of the $p$-adic valuation of an integer, where $p\in\mathbb{P}$. The $p$-adic valuation of an integer $n$, denoted by $\nu_{p}(n)$, is just the highest power of $p$ dividing $n$, i.e., $\nu_{p}(n):=\op{max}\{k\in\N:\;p^{k}\mid n\}$. We also make standard convention that $\nu_{p}(0)=+\infty$.

There are many papers devoted to the study of arithmetic properties of the sequence counting $m$-ary partitions. The case of $m=2$ was investigated by Euler. However, it seems that the first non-trivial result concerning this case was proved by Churchhouse \cite{Chu}. He proved that the sequence of the 2-adic valuations of $(b(n))_{n\in\N}$ is bounded by 2. More precisely, we have $b(0)=1, b(1)=1$ and if $n\geq 2$, we have $\nu_{2}(b(n))=2$ if and only if $n$ or $n-1$ can be written in the form $4^r(2u+1)$ for some $r\in\N_{+}$ and $u\in\N$. In the remaining cases the value of $\nu_{2}(b(n))$ is equal to 1. Churchhouse also stated some conjectures concerning the 2-adic behaviour of the difference $a_{2}(4n)-a_{2}(n)$. These conjectures were independently proved by R{\o}dseth \cite{Rod} and Gupta \cite{Gup} (see also the presentation given in Chapter 10 in the Andrews book). These results motivated the study of the divisibility of the difference $a_{m}(m^{r+1}n)-a_{m}(m^{r}n)$ by powers of $m$ (in the case of odd $m$) or by powers of $m/2$ (in the case of even $m$). See for example the papers of Gupta \cite{Gup1, Gup2}, Andrews \cite{And1}, R{\o}dseth and Sellers \cite{RodSel, RodSel1}.

Another kind of generalizations appeared in a recent work of Gawron, Miska and the first author \cite{GMU}. More precisely, for  $k\in\N_{+}$ one can define the sequence $(A_{m,k}(n))_{n\in\N}$, where
$$
F_{m}(x)^{k}=\prod_{n=0}^{\infty}\frac{1}{\left(1-x^{m^{n}}\right)^{k}}=\sum_{n=0}^{\infty}A_{m,k}(n)x^{n}.
$$
The sequence $(A_{m,k}(n))_{n\in\N}$, as the sequence $(a_{m}(n))_{n\in\N}$, can be interpreted in a natural combinatorial way. More precisely, the number $A_{m,k}(n)$ counts the number of representations of $n$ as sums of powers of $m$, where each summand has one among $k$ colors. Equivalently, we are looking for the number of solutions in non-negative integers of the system $\eps_{i}=\sum_{j=1}^{k}\eps_{i,j}$,  where $\eps_{i}$ comes from the $m$-ary representation given by (\ref{rep}).  In \cite{GMU}, the sequence of the 2-adic valuations of $A_{2,k}(n)$ was investigated. The authors proved that the sequence $(\nu_{2}(A_{2,k}(n)))_{n\in\N}$ with $k=2^{s}-1, s\in\N_{+}$, is bounded by 2. In fact the precise description whether  $\nu_{2}(A_{2,k}(n))=1$ or 2 is presented in the paper. Let $m=p$, where $p$ is a prime number. A question arises: is it possible to find a simple expression for an exponent $k$, such that the sequence $(\nu_{p}(A_{p,k}(n)))_{n\in\N}$ is bounded or even can be described in simple terms? This natural question was the main motivation for writing this paper. As we will see the required generalization can be given and is presented in Theorem \ref{mainthm}.

Let us describe the content of the paper in some details. In Section \ref{sec2} we present some preliminary result needed in the sequel and present the proof of the lemma which contains precise information on the $p$-adic valuation of certain sums which are closely related to the coefficients of the power series expansion of $F_{p}(x)^{-r}$, where $r\in\{1,\ldots,p-1\}$.

In Section \ref{sec3}, among other things we present the computation of the exact value of the $p$-adic valuation of the number $A_{p,(p-1)(p^s-1)}(n)$, where $p\in\mathbb{P}_{\geq 3}$ and $s\in\N_{+}$. More precisely, for $n\geq p^s$, we prove the equality $\nu_{p}(A_{p,(p-1)(p^s-1)}(n))=1$. Note that the number $k=(p-1)(p^s-1)$ for $p=2$ reduces to the exponent considered in \cite{GMU}. However, in the mentioned paper the sequence of the 2-adic valuations of $2^s-1$-colored partitions is not eventually constant (and it is not periodic). This shows striking difference between the case of $p=2$ and $p\geq 3$. Moreover, we obtain some result concerning the boundedness of the sequence $(\nu_{p}(A_{p,(p-1)(up^s-1)}(n)))_{n\in\N}$ for $p\in\mathbb{P}_{\geq 3}, u\in\{2,\ldots,p-1\}$ and $s\in\N_{+}$.


Finally, in the last section we state some conjectures on the shape of exponents $k$ which (according to numerical calculations) lead to bounded (but not-eventually constant) sequences $(\nu_{p}(A_{p,k}(n)))_{n\in\N}$.

\section{Preliminary results and the main lemma}\label{sec2}

Before we state our theorems we will need some notation and preliminary results concerning various objects related to unique representation of integer $n$ in the base $p$.

For a given $p$ (non-necessarily a prime), an integer $n$ and $i\in\{0,\ldots,p-1\}$ we define
$$
N_{p}(i,n)=|\{j:\;n=\sum_{j=0}^k \eps_j p^j,\eps_{j}\in\{0,\ldots,p-1\}\;\mbox{and}\;\eps_{j}=i\}|.
$$
The above number counts the number of the digits equal to $i$ in the base $p$ representation of the integer $n$. From the definition, we immediately deduce the following equalities:
\begin{equation}\label{equ1}
N_{p}(i,0)=0, \quad N_{p}(i,pn+j)=\begin{cases}
N_{p}(i,n),& \mbox{if}\;j\neq i\\
N_{p}(i,n)+1,& \mbox{if}\;j=i
\end{cases}.
\end{equation}

\begin{lem}\label{lem1}
Let $r\in\{1,\ldots,p-1\}$. We have
$$
F_{p}(x)^{-r}=\prod_{n=0}^{\infty}(1-x^{p^{n}})^{r}=\sum_{n=0}^{\infty}D_{p,r}(n)x^{n},
$$
where
\begin{equation}\label{equ2}
D_{p,r}(n)=\prod_{i=0}^{p-1}(-1)^{iN_{p}(i,n)}\binom{r}{i}^{N_{p}(i,n)},
\end{equation}
with the convention that $\binom{a}{b}=0$ for $b>a$ and $0^0=1$. Moreover, for $j\in\{0,\ldots,p-1\}$ and $n\in\N_{+}$ we have
$$
D_{p,r}(pn+j)=(-1)^{j}\binom{r}{j}D_{p,r}(n).
$$
\end{lem}
\begin{proof}
The first equality is a simple consequence of the uniqueness of the base $p$ representation of a given number $n$ and the equality
$$
(1-x^{p^{n}})^{r}=\sum_{j=0}^{r}\binom{r}{j}(-1)^{j}x^{p^{n}j}.
$$
Indeed, the number $D_{p,r}(n)$ can be written as
$$
D_{p,r}(n)=
\begin{cases}
\begin{array}{ll}
\prod_{i=0}^{p-1}(-1)^{iN_{p}(i,n)}\binom{r}{i}^{N_{p}(i,n)}, & \text{if $N_{p}(j,n)=0$ for all $j>r$} \\
0, & \text{if $N_{p}(j,n)\neq 0$ for some $j>r$}
\end{array}
\end{cases},
$$
and hence our compact formula.

In order to get the second part of our lemma we note that $N_{p}(i,n)=1$ for $n\in\{0,1\ldots,r\}$ and $N_{p}(i,n)=0$ for $n\in\{r+1\ldots,p-1\}$. With our convention we have that $D_{p,r}(n)=(-1)^{n}\binom{r}{n}$ in this case. Writing now $n=pn'+j$ for some $j\in\{0,\ldots,p-1\}$ we get
\begin{align*}
D_{p,r}(pn'+j)&=\prod_{i=0}^{p-1}(-1)^{iN_{p}(i,pn'+j)}\binom{r}{i}^{N_{p}(i,pn'+j)}\\
           &=\prod_{i\neq j}^{p-1}(-1)^{iN_{p}(i,n')}\binom{r}{i}^{N_{p}(i,n')}(-1)^{j(N_{p}(j,n')+1)}\binom{r}{j}^{N_{p}(j,n')+1}\\
           &=(-1)^{j}\binom{r}{j}\prod_{i=0}^{p-1}(-1)^{iN_{p}(i,n')}\binom{r}{i}^{N_{p}(i,n')}=(-1)^{j}\binom{r}{j}D_{p,r}(n'),
\end{align*}
and our lemma follows.
\end{proof}

The above results are true without the assumption $p\in\mathbb{P}$. However, from now up to the end of the paper we assume that $p$ is an odd prime. Moreover, in order to shorten the notation a bit we define the number $D_{p}(n)$ as
$$
D_{p}(n):=D_{p,p-1}(n).
$$
In particular $D_{p}(n)\neq 0$ for all $n\in\N$.

\begin{lem}\label{lem2}
Let $k\in\N_{+}$ and suppose that $p-1|k$. Then
$$
F_{p}(x)^{k}\equiv (1-x)^{\frac{k}{p-1}}\pmod{p^{\nu_{p}(k)+1}}.
$$
\end{lem}
\begin{proof}
Let us recall that if the congruence $a\equiv b\pmod{p}$ holds, then for any $t\in\N_{+}$ we have $a^{p^{t}}\equiv b^{p^{t}}\pmod{p^{t+1}}$.

From the definition of the function $F_{p}(x)$ we get
$$
(1-x)F_{p}(x)=\prod_{n=1}^{\infty}\frac{1}{1-x^{p^{n}}}=F_{p}(x^p),
$$
and thus we immediately deduce that $F_{p}(x)^{p-1}\equiv (1-x)\pmod{p}$. Invoking now the property mentioned above with $t=\nu_{p}(k)$ we get
$$
F_{p}(x)^{p^{\nu_{p}(k)}(p-1)}\equiv (1-x)^{p^{\nu_{p}(k)}}\pmod{p^{\nu_{p}(k)+1}}.
$$
Write $k=p^{\nu_{p}(k)}(p-1)\delta$. Raising the both sides of the last congruence to the power $\delta$ we get the result.

\end{proof}


We are ready to present the crucial lemma which will be the main tool in our study of the $p$-adic valuation of the number $A_{p,(p-1)(up^s-1)}(n)$ in the sequel. More precisely, the lemma contains information about behaviour of the $p$-adic valuation of the expression
$$
\sum_{i=0}^{u}(-1)^{i}\binom{u}{i}D_{p}(n-i).
$$

\begin{lem}\label{lem3}
Let $p\geq 3$ be prime and $u\in\{1,\ldots ,p-1\}$. Let $n\geq p$ be of the form $n=n''p^{s+1}+kp^{s}+j$ for some $n''\in\N, k\in\{1, \ldots, p-1\}, s\in\N_{+}$ and $j\in\{0,\ldots, p-1\}$. Then the following equality holds:
$$
\nu_{p}\left(\sum_{i=0}^{u}(-1)^{i}\binom{u}{i}D_{p}(n-i)\right)=\nu_{p}\left((p-k)\binom{p+u-1}{j}+k\binom{p+u-1}{p+j}\right).
$$
In particular:
\begin{enumerate}

\item[(a)] If $u=1$, then
$$
\nu_{p}\left(\sum_{i=0}^{u}(-1)^{i}\binom{u}{i}D_{p}(n-i)\right)=\nu_{p}(D_{p}(n)-D_{p}(n-1))=1,
$$
for any $n\in\N_{+}$.
\item[(b)] If $j\geq u$, then
we have the equality
$$
\nu_{p}\left(\sum_{i=0}^{u}(-1)^{i}\binom{u}{i}D_{p}(n-i)\right)=1.
$$
\item[(c)] If $u\geq 2$, then there exist $j, k\in\{0,\ldots,p-1\}, k\neq 0$, such that
we have
$$
\nu_{p}\left(\sum_{i=0}^{u}(-1)^{i}\binom{u}{i}D_{p}(n-i)\right)\geq 2.
$$

\end{enumerate}
\end{lem}
\begin{proof}
Let us suppose that $n\in\{u,\ldots,p-1\}$. Then
\begin{align*}
\sum_{i=0}^{u}(-1)^{i}\binom{u}{i}D_{p}(n-i)&=\sum_{i=0}^{u}(-1)^{i}\binom{u}{i}(-1)^{n-i}\binom{p-1}{n-i}\\
                                            &=(-1)^{n}\sum_{i=0}^{u}\binom{u}{i}\binom{p-1}{n-i}=(-1)^{n}\binom{p+u-1}{n}.
\end{align*}

Let $n\geq p$ and write $n=pn'+j$ for some $j\in\{0,\ldots,p-1\}$ and $n'\in\N_{+}$. We have
\begin{align*}
\sum_{i=0}^{u}&(-1)^{i}\binom{u}{i}D_{p}(n-i)= \sum_{i=0}^{u}(-1)^{i}\binom{u}{i}D_{p}(pn'+j-i) \\
 = & \sum_{i=0}^{j}(-1)^{i}\binom{u}{i}D_{p}(pn'+j-i)+\sum_{i=j+1}^{u}(-1)^{i}\binom{u}{i}D_{p}(p(n'-1)+p+j-i) \\
 = & \sum_{i=0}^{j}(-1)^{i}\binom{u}{i}(-1)^{j-i}\binom{p-1}{j-i}D_{p}(n')+\sum_{i=j+1}^{u}(-1)^{i}\binom{u}{i}(-1)^{p+j-i}\binom{p-1}{p+j-i}D_{p}(n'-1) \\
  = & (-1)^{j}D_{p}(n')\sum_{i=0}^{j}\binom{u}{i}\binom{p-1}{j-i}+(-1)^{p+j}D_{p}(n'-1)\sum_{i=j+1}^{u}\binom{u}{i}\binom{p-1}{p+j-i} \\
  = & (-1)^{j}D_{p}(n')\binom{p+u-1}{j}-(-1)^{j}D_{p}(n'-1)\binom{p+u-1}{p+j} \\
  = & (-1)^{j}\left[D_{p}(n')\binom{p+u-1}{j}-D_{p}(n'-1)\binom{p+u-1}{p+j}\right].
\end{align*}
Write $n'=n''p+k$. If $k\geq 1$, then
\begin{align*}
D_{p}(n')\binom{p+u-1}{j}&-D_{p}(n'-1)\binom{p+u-1}{p+j}\\
                         &=  D_{p}(pn''+k)\binom{p+u-1}{j}-D_{p}(pn'+k-1)\binom{p+u-1}{p+j} \\
                         &=  D_{p}(n'')(-1)^{k}\left[\binom{p-1}{k}\binom{p+u-1}{j}+\binom{p-1}{k-1}\binom{p+u-1}{p+j}\right].
\end{align*}
If $k=0$, then
\begin{align*}
D_{p}(n')&\binom{p+u-1}{j}-D_{p}(n'-1)\binom{p+u-1}{p+j}\\\
         &=  D_{p}(pn'')\binom{p+u-1}{j}-D_{p}(p(n''-1)+p-1)\binom{p+u-1}{p+j} \\
         &=  D(n'')\binom{p+u-1}{j}-D_{p}(n''-1)\binom{p+u-1}{p+j}.
\end{align*}
Hence if $n'=p^{s}(ap+k)$ with $k\geq 1$, then
\begin{align*}
D_{p}(n')&\binom{p+u-1}{j}-D_{p}(n'-1)\binom{p+u-1}{p+j}\\
         &=D_{p}(a)(-1)^{k}\left[\binom{p-1}{k}\binom{p+u-1}{j}+\binom{p-1}{k-1}\binom{p+u-1}{p+j}\right].
\end{align*}
If $n'=kp^{s}$ with $1<k\leq p-1$, then
\begin{align*}
D_{p}(n')&\binom{p+u-1}{j}-D_{p}(n'-1)\binom{p+u-1}{p+j}\\
         &= D_{p}(k)\binom{p+u-1}{j}-D_{p}(k-1)\binom{p+u-1}{p+j} \\
         &= (-1)^{k}\left[\binom{p-1}{k}\binom{p+u-1}{j}+\binom{p-1}{k-1}\binom{p+u-1}{p+j}\right].
\end{align*}
Finally, for $n'=p^{s}$, we get
\begin{align*}
D_{p}(n')&\binom{p+u-1}{j}-D_{p}(n'-1)\binom{p+u-1}{p+j}\\
         &= D_{p}(p)\binom{p+u-1}{j}-D_{p}(p-1)\binom{p+u-1}{p+j} \\
         &= -(p-1)\binom{p+u-1}{j}-\binom{p+u-1}{p+j} \\
         &= -\left[\binom{p-1}{1}\binom{p+u-1}{j}+\binom{p-1}{1-1}\binom{p+u-1}{p+j}\right].
\end{align*}
Thus, in any case we get that if $n=n''p^{s+1}+kp^{s}+j$ for some $n''\in\N$, $k\in\{1,\ldots ,p-1\}$, $s\in\N_{+}$ and $j\in\{0,\ldots ,p-1\}$, then
$$
\nu_{p}\left(\sum_{i=0}^{u}(-1)^{i}\binom{u}{i}D_{p}(n-i)\right)=\nu_{p}\left(\binom{p-1}{k}\binom{p+u-1}{j}+\binom{p-1}{k-1}\binom{p+u-1}{p+j}\right).
$$
The equality
\begin{align*}
\binom{p-1}{k}&\binom{p+u-1}{j}+\binom{p-1}{k-1}\binom{p+u-1}{p+j}\\
              &= \frac{p-k}{k}\binom{p-1}{k-1}\binom{p+u-1}{j}+\binom{p-1}{k-1}\binom{p+u-1}{p+j} \\
              &= \frac{\binom{p-1}{k-1}}{k}\left[(p-k)\binom{p+u-1}{j}+k\binom{p+u-1}{p+j}\right]
\end{align*}
implies
$$
\nu_{p}\left(\sum_{i=0}^{u}(-1)^{i}\binom{u}{i}D_{p}(n-i)\right)=\nu_{p}\left((p-k)\binom{p+u-1}{j}+k\binom{p+u-1}{p+j}\right).
$$
This finishes the proof of the main part of our lemma.

In order to get the part (a), it is enough to observe that for $u=1$ we have
$$
(p-k)\binom{p+u-1}{j}+k\binom{p+u-1}{p+j}=
\begin{cases}\begin{array}{ll}
(p-k)\binom{p}{j}, & \text{for } j\geq 1 \\
p, & \text{for } j=0
\end{array}
\end{cases}.
$$
The result follows since $1\leq k\leq p-1$.

If $u$ is arbitrary and $j\geq u$, then
\begin{align*}
\nu_{p}\left((p-k)\binom{p+u-1}{j}+k\binom{p+u-1}{p+j}\right)&=\nu_{p}\left((p-k)\binom{p+u-1}{j}\right)\\
                                  &=\nu_{p}\left(\binom{p+u-1}{j}\right)=1.
\end{align*}
Hence we get the part (b).

In order to prove the part (c) let us fix $p$ and $u$. We want to find $j$ and $k$ such that
\begin{equation}\label{spec}
\nu_{p}\left((p-k)\binom{p+u-1}{j}+k\binom{p+u-1}{p+j}\right)\geq 2.
\end{equation}
Let us write
$$
(p-k)\binom{p+u-1}{j}+k\binom{p+u-1}{p+j}=p\binom{p+u-1}{j}-k\left(\binom{p+u-1}{j}-\binom{p+u-1}{p+j}\right).
$$
Thus, we only need to find  $j\leq u-1$ and $k\in\{1,\ldots ,p-1\}$ such that
$$
k\left(\binom{p+u-1}{j}-\binom{p+u-1}{p+j}\right)\equiv p\binom{p+u-1}{j}\pmod{p^{2}},
$$
or equivalently
$$
k\frac{\binom{p+u-1}{j}-\binom{p+u-1}{p+j}}{p}\equiv \binom{p+u-1}{j}\pmod p.
$$
In consequence, the sufficient condition for the inequality (\ref{spec}) to hold, is the condition
$$
p^{2}\nmid \binom{p+u-1}{j}-\binom{p+u-1}{p+j}.
$$
We prove that this is satisfied in at least one case $j=0$ or $j=1$. Let us observe that if $j=0$, then
\begin{align*}
\binom{p+u-1}{0}-\binom{p+u-1}{p}= & 1-\frac{(p+u-1)!}{p!(u-1)!}=1-\frac{(p+1)\ldots (p+u-1)}{(u-1)!} \\
 = & \frac{1}{(u-1)!}\left((u-1)!-(p+1)\ldots (p+u-1)\right) \\
 \equiv & \frac{1}{(u-1)!}\bigg((u-1)!-\bigg((u-1)!+p\bigg(\sum_{i=1}^{u-1}\prod_{\substack{j=1 \\ j\neq i}}^{u-1}j\bigg)\bigg)\bigg) \\
 = &-\frac{p}{(u-1)!}\bigg(\sum_{i=1}^{u-1}\prod_{\substack{j=1 \\ j\neq i}}^{u-1}j\bigg)=-p\left(\sum_{i=1}^{u-1}\frac{1}{i}\right)\pmod{p^{2}}.
\end{align*}
Similarly, for $j=1$ we have
$$
\binom{p+u-1}{1}-\binom{p+u-1}{p+1}\equiv -p(u-1)(p+u-1)\left(\sum_{i=2}^{u-2}\frac{1}{i}\right)\pmod{p^{2}}.
$$
Thus, if $p^{2}|\binom{p+u-1}{j}-\binom{p+u-1}{p+j}$ for $j=0$ and $j=1$, then $p^{2}$ divides also the number
$$
K:=(u-1)(p+u-1) \left(\binom{p+u-1}{0}-\binom{p+u-1}{p}\right)-\left(\binom{p+u-1}{1}-\binom{p+u-1}{p+1}\right).
$$
However,
\begin{align*}
 K \equiv & -p(u-1)(p+u-1)\left(\sum_{i=1}^{u-1}\frac{1}{i}-\sum_{i=2}^{u-2}\frac{1}{i}\right) \\
 = & -p(u-1)(p+u-1)\left(1+\frac{1}{u-1}\right)=-pu(p+u-1) \pmod{p^{2}},
\end{align*}
and we get a contradiction. The result follows.
\end{proof}

\section{The results}\label{sec3}

Now, we are ready to prove the following

\begin{thm}\label{mainthm}
Let $p\in\mathbb{P}_{\geq 3}$, $u\in\{1,\ldots ,p-1\}$ and $s\in\N_{+}$.
\begin{enumerate}
\item[(a)] If $n>up^{s}$, then
$$
\nu_{p}(A_{p,(p-1)(up^s-1)}(n))\geq 1.
$$
\item[(b)] If $n>p^{s}$, then
$$
\nu_{p}(A_{p,(p-1)(p^s-1)}(n))= 1.
$$
\item[(c)] If $u\geq 2$, then
$$
\nu_{p}(A_{p,(p-1)(up^s-1)}(n))= 1
$$
for infinitely many $n$.
\item[(d)] If $u\geq 2$, then
$$
\nu_{p}(A_{p,(p-1)(up^s-1)}(n))\geq 2
$$
for infinitely many $n$.
\item[(e)] If $s\geq 2$ and $n\geq p^{s+1}$ with the unique base $p$-representation $n=\sum_{i=0}^{v}\eps_{i}p^{i}$ and
$$
\nu_{p}(A_{p,(p-1)(up^s-1)}(n))\in\{1,2\},
$$
then the value of $\nu_{p}(A_{p,(p-1)(up^s-1)}(n))$ depends only on the coefficient $\eps_{s}$ and the first non-zero coefficient $\eps_{t}$ with $t>s$.
\item[(f)] If $s\geq 2$ and
$$
\nu_{p}(A_{p,(p-1)(up^s-1)}(n))\leq s
$$
for $n>up^{s}$, then also
$$
\nu_{p}(A_{p,(p-1)(up^s-1)}(pn))=\nu_{p}(A_{p,(p-1)(up^s-1)}(pn+1))=\ldots =\nu_{p}(A_{p,(p-1)(up^s-1)}(pn+(p-1))).
$$
\end{enumerate}
\end{thm}
\begin{proof}
Applying Lemma \ref{lem2} with $k=p^{s}(p-1)$ we get
$$
F_{p}(x)^{p^{s+1}-p^{s}}\equiv (1-x)^{\frac{p^{s+1}-p^{s}}{p-1}}\pmod{p^{s+1}}.
$$
Raising both sides to the power $u$ we get
$$
F_{p}(x)^{up^{s}(p-1)}\equiv (1-x)^{up^{s}}\pmod{p^{s+1}}.
$$
Multiplying both sides of the above congruence by $F_{p}(x)^{1-p}$ we get
$$
F_{p}(x)^{(up^{s}-1)(p-1)}\equiv (1-x)^{up^{s}}F_{p}(x)^{1-p}\pmod{p^{s+1}}.
$$
Equivalently, by using Lemma \ref{lem1} and comparing coefficients on both sides of the above congruence, we get the congruence
$$
A_{p,(p-1)(up^{s}-1)}(n)\equiv \sum_{i=0}^{\op{min}\{n,up^{s}\}}(-1)^{i}\binom{up^{s}}{i}D_{p}(n-i)\pmod{p^{s+1}}.
$$
Let us take $n>up^{s}$ and write
$$
n=\sum_{i=0}^{t}\eps_{i}p^{i}
$$
for some $t\geq s$ and $n_{s}\geq u$. Let us define the numbers $n_{0},\ldots ,n_{t}$ recursively in the following way
\begin{align*}
n_{0}:= & n, \\
n_{i+1}:= & \frac{n_{i}-\eps_{i}}{p}
\end{align*}
for $i=0,1,\ldots,t-1$.

We start by proving the part (f). Let $s\geq 2$. Observe, that the formula
\begin{equation}\label{equbinom}
k\binom{p^{s}}{k}=p^{s}\binom{p^{s}-1}{k-1}\equiv 0\pmod{p^{s}}
\end{equation}
implies that if $p\nmid k$, then $p^{s}|\binom{p^{s}}{k}$. Hence, using Lemma \ref{lem1}, we get
\begin{align*}
\sum_{i=0}^{up^{s}}&(-1)^{i}\binom{up^s}{i}D_{p}(n-i)\\
  &=  \sum_{i=0}^{up^{s}-1}(-1)^{i}\binom{up^s}{i}D_{p}(n-i)-D_{p}(n-up^{s}) \\
  &\equiv \sum_{i=0}^{up^{s-1}-1}(-1)^{ip^{s-1}}\binom{up^s}{ip}D_{p}(n-ip)-D_{p}(n-up^{s}) \\
  &=  \sum_{i=0}^{up^{s-1}-1}(-1)^{i}\binom{up^s}{ip}D_{p}((n_{1}-i)p+\eps_{0})-D_{p}((n_{1}-up^{s-1})p+\eps_{0}) \\
  &=  (-1)^{\eps_{0}}\binom{p-1}{\eps_{0}}\left( \sum_{i=0}^{up^{s-1}-1}(-1)^{i}\binom{up^s}{i}D_{p}(n_{1}-i)-D_{p}(n_{1}-up^{s-1}) \right)\\
  &=  (-1)^{\eps_{0}}\binom{p-1}{\eps_{0}}\left( \sum_{i=0}^{up^{s-1}}(-1)^{i}\binom{up^s}{i}D_{p}(n_{1}-i) \right)\pmod{p^{s}}.
\end{align*}
Thus, if $\nu_{p}(A_{p,(p-1)(up^s-1)}(n))\leq s$, then
$$
\nu_{p}\left(A_{p,(p-1)(up^s-1)}(n)\right)=\op{min}\left\{s,\nu_{p}\left(\sum_{i=0}^{up^{s-1}}(-1)^{i}\binom{up^s}{i}D_{p}(n_{1}-i)\right)\right\},
$$
which does not depend on $\eps_{0}$. Hence we get (f).

In order to prove parts (a) to (e) it is enough to show that there exists a constant $C_{1}$ not divisible by $p$, such that
$$
A_{p,(p-1)(up^s-1)}(n)\equiv C_{1}\left(\sum_{i=0}^{u}(-1)^{i}\binom{u}{i}D_{p}(n-i)\right)\pmod{p^{2}},
$$
and then apply Lemma \ref{lem3}.

From the formula (\ref{equbinom}) we see that the condition $p^{s-1}\nmid k$ implies $p^{2}|\binom{p^{s}}{k}$. We proceed similarly as in the proof of the part (f). Indeed, we have
\begin{align*}
\sum_{i=0}^{up^{s}}(-1)^{i}\binom{up^s}{i} & D_{p}(n-i)\equiv C_{2}\left( \sum_{i=0}^{up}(-1)^{i}\binom{up^s}{ip^{s-1}}D_{p}(n_{s-1}-i) \right)\pmod{p^{2}},
\end{align*}
with
$$
C_{2}=\prod_{i=0}^{s-2}(-1)^{\eps_{i}}\binom{p-1}{\eps_{i}}.
$$
In the case of $s=1$, the starting sum is of the same form with $C_{2}=1$. It is clear that $C_{2}\not\equiv 0\pmod p$. We need to investigate the expression in the brackets.

Observe that if $p$ does not divide $i$, then $p|\binom{up^{s}}{ip^{s-1}}$. We also have
$$
D_{p}(n)\equiv 1\pmod p
$$
and
$$
\binom{up^{s}}{ip^{s}}\equiv\binom{u}{i}\pmod p.
$$
The last equivalence is a consequence of Lucas theorem.
Hence
\begin{align*}
\sum_{i=0}^{up}(-1)^{i}&\binom{up^s}{ip^{s-1}}D_{p}(n_{s-1}-i)\\
&=  \sum_{\substack{i=0 \\ p\nmid i}}^{up}(-1)^{i}\binom{up^s}{ip^{s-1}}D_{p}(n_{s-1}-i)+\sum_{i=0}^{u}(-1)^{ip}\binom{up^s}{ip^{s}}D_{p}(n_{s-1}-ip) \\
 &=  p\sum_{\substack{i=0 \\ p\nmid i}}^{up}(-1)^{i}\frac{1}{p}\binom{up^s}{ip^{s-1}}D_{p}(n_{s-1}-i)+\sum_{i=0}^{u}(-1)^{i}\binom{up^s}{ip^{s}}D_{p}(n_{s-1}-ip) \\
 &\equiv  p\sum_{\substack{i=0 \\ p\nmid i}}^{up}(-1)^{i}\frac{1}{p}\binom{up^s}{ip^{s-1}}+\sum_{i=0}^{u}(-1)^{i}\binom{up^s}{ip^{s}}D_{p}(n_{s-1}-ip) \\
 &=  \sum_{\substack{i=0 \\ p\nmid i}}^{up}(-1)^{i}\binom{up^s}{ip^{s-1}}+\sum_{i=0}^{u}(-1)^{i}\binom{up^s}{ip^{s}}D_{p}(n_{s-1}-ip) \\
 &=  \sum_{i=0}^{up}(-1)^{i}\binom{up^s}{ip^{s-1}}+\sum_{i=0}^{u}(-1)^{i}\binom{up^s}{ip^{s}}(D_{p}(n_{s-1}-ip)-1) \\
 &\equiv  \sum_{i=0}^{up^{s}}(-1)^{i}\binom{up^s}{i}+p\sum_{i=0}^{u}(-1)^{i}\binom{up^s}{ip^{s}}\frac{D_{p}(n_{s-1}-ip)-1}{p} \\
 &\equiv  0+ p\sum_{i=0}^{u}(-1)^{i}\binom{u}{i}\frac{D_{p}(n_{s-1}-ip)-1}{p} \\
 &=  \sum_{i=0}^{u}(-1)^{i}\binom{u}{i}(D_{p}(n_{s-1}-ip)-1) \\
 &=  \sum_{i=0}^{u}(-1)^{i}\binom{u}{i}D_{p}(n_{s-1}-ip)-\sum_{i=0}^{u}(-1)^{i}\binom{u}{i} \\
 &=  \sum_{i=0}^{u}(-1)^{i}\binom{u}{i}D_{p}((n_{s}-i)p+\eps_{s-1})-0 \\
 &=  (-1)^{\eps_{s-1}}\binom{p-1}{\eps_{s-1}}\sum_{i=0}^{u}(-1)^{i}\binom{u}{i}D_{p}(n_{s}-i) \pmod{p^{2}}. \\
\end{align*}
Thus, we get the congruence
$$
A_{p,(p-1)(up-1)}(n)\equiv C_{1} \left(\sum_{i=0}^{u}(-1)^{i}\binom{u}{i}D_{p}(n_{s}-i)\right)\pmod{p^{2}}
$$
with
$$
C_{1}=\prod_{i=0}^{s-1}(-1)^{\eps_{i}}\binom{p-1}{\eps_{i}}\not\equiv 0\pmod p.
$$
The result now follows from Lemma \ref{lem2}.
\end{proof}

We have already proved a quite strong result concerning the behaviour of the $p$-adic valuation of $A_{p,k}(n)$ for certain exponents $k$. The question arises: can we prove something in the opposite direction? More precisely: is it possible, given any $s\in\N_{+}$, to construct an exponent $k\in\N_{+}$ such that the congruence $A_{p,k}(n)\equiv 0\pmod{p^{s}}$ has infinitely many solutions?

We are ready to investigate more closely the set
$$
\mathcal{E}_{p}:=\{s\in\N:\ \text{there exists $k\in\N$ such that } A_{p,k}(n)\equiv 0\pmod{p^{s}}\ \text{for infinitely many $n\in\N$}\}.
$$

From Lemma \ref{lem2} we immediately deduce that $\mathcal{E}_{p}=\N$. However, for $s\in\N_{+}$, the exponents $k$ in this case are divisible by $p$. Can we produce exponents without this property? The answer to this question is contained in the following

\begin{thm}\label{mainthm2}
Let $k\in\N_{+}, p\in\mathbb{P}_{\geq 3}$ and suppose that $p^2(p-1)|k$ and $r\in\{1,\ldots ,p-2\}$.

Then, there are infinitely many $n\in\N_{+}$ such that
$$
\nu_{p}(A_{p,k-r}(n))\geq \nu_{p}(k).
$$
\end{thm}
\begin{proof}

From the assumption $p-1|k$ and the congruence $F_{p}(x)^{p-1}\equiv 1-x\pmod{p}$ we get
$$
F_{p}(x)^{k}\equiv (1-x)^{\frac{k}{p-1}}\pmod{p^{\nu_{p}(k)+1}}.
$$
Multiplying both sides of the above congruence by $F(x)^{-r}$ we obtain
$$
F_{p}(x)^{k-r}\equiv (1-x)^{\frac{k}{p-1}}F_{p}(x)^{-r}\pmod{p^{\nu_{p}(k)+1}}.
$$
Write $k/(p-1)=\delta p^{\nu}$ for some $\delta\in\N_{+}$, where $\nu=\nu_{p}(k)\geq 2$. Hence, comparing coefficients on both sides of the above congruence we get the congruence
$$
A_{p,k-r}(n)\equiv \sum_{i=0}^{\op{min}\{n,\delta p^{\nu}\}}(-1)^{i}\binom{\delta p^{\nu}}{i}D_{p,r}(n-i)\pmod{p^{\nu+1}}.
$$
Let $n>\delta p^{\nu}$ and write $n=pn'+j$ for some $n'\in\N$ and $j\in\{0,\ldots,p-1\}$. Let us observe that if $p\nmid i$, then $p^{\nu}|\binom{\delta p^{\nu}}{i}$. Thus
\begin{align*}
A_{p,k-r}(n)\equiv & \sum_{i=0}^{\delta p^{\nu}}(-1)^{i}\binom{\delta p^{\nu}}{i}D_{p,r}(n-i) \\
 \equiv & \sum_{i=0}^{\delta p^{\nu -1}}(-1)^{i}\binom{\delta p^{\nu}}{ip}D_{p,r}(n-ip) \\
 = & \sum_{i=0}^{\delta p^{\nu -1}}(-1)^{i}\binom{\delta p^{\nu}}{ip}D_{p,r}((n'-i)p+j) \\
 = & (-1)^{j}\binom{r}{j}\sum_{i=0}^{\delta p^{\nu -1}}(-1)^{i}\binom{\delta p^{\nu}}{ip}D_{p,r}(n'-i) \pmod{p^{\nu}}.
\end{align*}
In consequence, $A_{p,k-r}(n)\equiv 0\pmod{p^{\nu}}$ for any $n>\delta p^{\nu}=k/(p-1)$ satisfying $n\equiv j\pmod p$, where $j\in\{r+1,\ldots,p-1\}$ is arbitrarily chosen. The proof is complete.
\end{proof}

\begin{rem}
{\rm The above result can be seen as a (weak) generalization of the result obtained by R{\o}dseth and Sellers in \cite{RodSel}, which says that for any odd $m$ and $s\in\N_{+}$, the congruence $a_{m}(n)\equiv 0\pmod{m^{s}}$ has infinitely many solutions in $n$. The weakness of our result is in the dependence of the number of colors and the exponent~$s$.}
\end{rem}

From Lemma \ref{lem2} and the proof of our theorem we can easily deduce the following

\begin{cor}
Let $k\in\N_{+}$ and suppose that the sequence $(\nu_{p}(A_{p,k}(n)))_{n\in\N}$ is eventually constant and equal to $1$.
\begin{enumerate}
\item[(a)] If $p-1|k$, then $p\nmid k$.
\item[(b)] If $p-1\nmid k$, $k=(p-1)k'+q$ for some $q\in\{1,\ldots ,p-2\}$, and $p|k'+1$, then $p^{2}\nmid k'+1$.
\end{enumerate}
\end{cor}

\section{Two conjectures and a question}\label{sec4}

Our Theorem \ref{mainthm} implies boundedness of the sequence $(\nu_{p}(A_{p,k}(n)))_{n\in\N}$ with $k=(p-1)(p^s-1)$ and $s\in\N_{+}$. It is natural to ask whether the founded family of exponents is the only one with this property. In order to get some feelings about this question, we performed numerical search with primes $p\in\{3,5,7,11\}$, the exponents $k\leq 100$ and $n\leq 10^4$. Our numerical search motivated the following

\begin{conj}
Let $p\in\mathbb{P}_{\geq 3}, u\in\{2,\ldots,p-1\}$ and $s\in\N_{+}$. Then, for $n\geq up^{s}$ we have
$$
\nu_{p}(A_{p,(p-1)(up^s-1)}(n))\in\{1,2\}.
$$
Moreover, for each $n\in\N_{+}$ we have the equalities
$$
\nu_{p}(A_{p,(p-1)(up^s-1)}(pn))=\nu_{p}(A_{p,(p-1)(up^s-1)}(pn+1))=\ldots=\nu_{p}(A_{p,(p-1)(up^s-1)}(pn+p-1)).
$$
\end{conj}

The above conjecture was suggested by our numerical observations. Moreover, the form of the conjecture motivated the parts (c) to (f) in Theorem \ref{mainthm}.

Let $k\in\N_{\geq 2}$ be given. We say that the sequence ${\bf \eps}=(\eps_{n})_{n\in\N}$ is $k$-automatic  if and only if the following set
$$
K_{k}({\bf \eps})=\{(\eps_{k^{i}n+j})_{n\in\N}:\;i\in\N\;\mbox{and}\;0\leq j<k^{i}\},
$$
called the $k$-kernel of ${\bf \eps}$, is finite. In the case of $p=2$ we know that the sequence $(\nu_{2}(A_{2,2^s-1}(n)))_{n\in\N}$ is 2-automatic (and it is not eventually periodic). In Theorem \ref{mainthm} we proved that the sequence $(\nu_{p}(A_{p,k}(n)))_{n\in\N}$ for $k=(p-1)(p^s-1)$ with $p\geq 3$, is eventually constant and hence $k$-automatic for any $k$.

We calculated the first $10^5$ elements of the sequence $(\nu_{p}(A_{p,(p-1)(up^s-1)}(n)))_{n\in\N}$ for any $p\in\{3, 5, 7\}, s\in\{1,2\}$ and $u\in\{1,\ldots,p-1\}$ and were not able to spot any general relations. Our numerical observations lead us to the following

\begin{ques}
For which $p\in\mathbb{P}_{\geq 5}, s\in\N$ and $u\in\{2,\ldots, p-1\}$, the sequence $(\nu_{p}(A_{p,(p-1)(up^s-1)}(n)))_{n\in\N}$ is $k$-automatic for some $k\in\N_{+}$?
\end{ques}

Finally, we formulate the following

\begin{conj}
Let $k\in\N_{+}, p\in\mathbb{P}$ and suppose that $k$ is not of the form $(p-1)(up^s-1)$ for $s\in\N$ and $u\in\{1,\ldots,p-1\}$. Then, the sequence
$(\nu_{p}(A_{p,k}(n)))_{n\in\N}$ is unbounded.
\end{conj}

\bigskip

\noindent {\bf Acknowledgments.}
The authors express their gratitude to the referee for careful reading of the manuscript
and valuable suggestions, which improved the quality of the paper.

\vskip 1cm

\noindent Maciej Ulas, Jagiellonian University, Faculty of Mathematics and Computer Science, Institute of
Mathematics, {\L}ojasiewicza 6, 30-348 Krak\'ow, Poland; email:
maciej.ulas@uj.edu.pl

\bigskip

\noindent B{\l}a\.{z}ej \.{Z}mija, Jagiellonian University, Faculty of Mathematics and Computer Science, Institute of
Mathematics, {\L}ojasiewicza 6, 30-348 Krak\'ow, Poland; email:
blazejz@poczta.onet.pl

\end{document}